\documentclass[UTF-8,reqno]{amsart}
\usepackage{enumerate}
\usepackage{mhequ, tikz, wasysym, stmaryrd}
\usepackage{amssymb,url,color, booktabs,nccmath}
\usepackage[left=3.2cm,right=3.2cm,top=4cm,bottom=4cm]{geometry}
\usepackage{mathrsfs}
\usepackage{enumitem,dsfont}

\usepackage{color}
\usepackage[colorlinks=true]{hyperref}
\hypersetup{
	linkcolor=blue,          
	citecolor=red,        
	filecolor=blue,      
	urlcolor=cyan
}

\definecolor{darkergreen}{rgb}{0.0, 0.5, 0.0}

\numberwithin{equation}{section}
\def\theequation{\arabic{section}.\arabic{equation}}
\newcommand{\be}{\begin{eqnarray}}
	\newcommand{\ee}{\end{eqnarray}}
\newcommand{\ce}{\begin{eqnarray*}}
	\newcommand{\de}{\end{eqnarray*}}
\newtheorem{theorem}{Theorem}[section]
\newtheorem{lemma}[theorem]{Lemma}
\newtheorem{proposition}[theorem]{Proposition}
\newtheorem{Examples}[theorem]{Example}
\newtheorem{corollary}[theorem]{Corollary}

\newtheorem{definition}[theorem]{Definition}
\theoremstyle{definition}
\newtheorem{remark}[theorem]{Remark}

\def\eps{\varepsilon}

\def\u{\mathbf{u}}
\def\w{\mathrm{w}}
\def\p{\partial}

\def\<{{\langle}}
\def\>{{\rangle}}
\def\({{\Big(}}
\def\){{\Big)}}

\def\bx{{\mathbf{x}}}

\def\dif{{\mathord{{\rm d}}}}

\def\no{\nonumber}
\def\={&\!\!=\!\!&}

\newcommand{\eqdef}{\stackrel{\mbox{\tiny def}}{=}}

\def\cA{{\mathcal A}}

\def\cJ{{\mathcal J}}

\def\cS{{\mathcal S}}

\def\mC{{\mathbb C}}

\def\mL{{\mathbb L}}

\def\mN{{\mathbb N}}

\def\mR{{\mathbb R}}
\def\mS{{\mathbb S}}
\def\mT{{\mathbb T}}

\def\1{{\mathbf{1}}}

\def\sC{{\mathscr C}}

\def\sF{{\mathscr F}}

\def\sI{{\mathscr I}}

\def\sL{{\mathscr L}}

\def\geq{\geqslant}
\def\leq{\leqslant}

\def\eps{\varepsilon}

\def\u{\mathbf{u}}
\def\w{\mathrm{w}}
\def\p{\partial}

\def\<{{\langle}}
\def\>{{\rangle}}
\def\({{\Big(}}
\def\){{\Big)}}

\def\bx{{\mathbf{x}}}

\def\dif{{\mathord{{\rm d}}}}

\def\no{\nonumber}
\def\={&\!\!=\!\!&}
\def\bt{\begin{theorem}}
	\def\et{\end{theorem}}
\def\bl{\begin{lemma}}
	\def\el{\end{lemma}}
\def\br{\begin{remark}}
	\def\er{\end{remark}}
\def\bx{\begin{Examples}}
	\def\ex{\end{Examples}}
\def\bd{\begin{definition}}
	\def\ed{\end{definition}}
\def\bp{\begin{proposition}}
	\def\ep{\end{proposition}}
\def\bc{\begin{corollary}}
	\def\ec{\end{corollary}}

\def\geq{\geqslant}
\def\leq{\leqslant}

 \def\R{\mathbb R}
 \def\R{\mathbb R}

\def\<{\langle} \def\>{\rangle}

\allowdisplaybreaks

\begin{document}

\title[Global well-posedness for 2D gPAM]{Global well-posedness for 2D generalized Parabolic Anderson Model via paracontrolled calculus}

\author{Hao Shen}
\address[H. Shen]{Department of Mathematics, University of Wisconsin - Madison, USA}
\email{pkushenhao@gmail.com}

\author{Rongchan Zhu}
\address[R. Zhu]{Department of Mathematics, Beijing Institute of Technology, Beijing 100081, China 
}
\email{zhurongchan@126.com}

\author{Xiangchan Zhu}
\address[X. Zhu]{ Academy of Mathematics and Systems Science,
	Chinese Academy of Sciences, Beijing 100190, China
}
\email{zhuxiangchan@126.com}

\dedicatory{(In memory of Giuseppe Da Prato.)}

\maketitle

\begin{abstract}
	This article revisits the problem of global well-posedness for the generalized parabolic Anderson model on $\mathbb{R}^+\times \mathbb{T}^2$ within the framework of paracontrolled calculus \cite{GIP15}. The model is given by the equation:
	\begin{equation*}
		(\partial_t-\Delta) u=F(u)\eta
	\end{equation*}
	where $\eta\in C^{-1-\kappa}$ with $1/6>\kappa>0$, and $F\in C_b^2(\mathbb{R})$. Assume that $\eta\in C^{-1-\kappa}$ and can be lifted to enhanced noise, we derive new a priori bounds. The key idea follows from the recent work
	\cite{CFW24}  by A.Chandra, G.L. Feltes and H.Weber  to represent the leading error term as a transport type term, and our  techniques encompass the paracontrolled calculus, the maximum principle, and the localization approach (i.e. high-low frequency argument).
\end{abstract}

\setcounter{tocdepth}{2}
\tableofcontents

\section{Introduction}
In this paper we consider the following equation, known as the generalized parabolic Anderson model (gPAM), on $\mR^+\times \mT^2$:
\begin{equation}\label{equ}
	\aligned
	\sL u&=F(u)\eta\;,
	\\u(0)&=u_0\;,
	\endaligned
\end{equation}
where $\sL=\p_t-\Delta$,  $\eta\in C_TC^{-1-\kappa}$ with $1/6>\kappa>0$ and $F\in C_b^2(\mR)$. The presence of the singularity in $\eta$ renders the nonlinear term $F(u)\eta$ analytically ill-defined. More precisely, in
view of Schauder's estimates $u$ is expected to be two degrees of regularity better, i.e. $ C^{1-\kappa}$,
and consequently, the product $F(u)\eta$ is not well-defined  in the classical sense.

This lack of regularity problem arises in a broad class of singular SPDEs.
The pioneering work in solving such problems and giving a notion of strong solutions to these equations
is attributed to Da Prato and Debussche for the 2D Navier--Stokes equation driven by space--time white noise \cite{DD02},
as well as the 2D dynamical $\Phi^4$ equation \cite{DD03}. In these landmark papers, they developed a type of argument, now known as the Da Prato--Debussche method, which identifies the most singular parts of the solutions to these equations and
re-interpret the singular nonlinear term by renormalization techniques.
This method is later applied to numerous settings, and extended in various ways. For instance, for the model \eqref{equ} with $\eta$ being the spatial white noise and $F(u)=u$, \cite{HL2d} obtained well-posedness by a very simple argument of Da Prato--Debussche type. Note that the method  has also been applied to elliptic and hyperbolic equations; we refer to \cite[Section~4.1]{CS20} for more references that were directly influenced by the above seminal work of Da Prato and Debussche.

The more irregular setting, as exemplified by \eqref{equ}, remained open for a much longer period, demanding substantially new ideas. These innovative ideas, concepts and theories emerged by the time of the introduction of the theory of regularity structures by Hairer \cite{Hai14} and the  paracontrolled calculus by Gubinelli, Imkeller, and Perkowski \cite{GIP15}.  Based on these theories
 \eqref{equ} can be interpreted in a renormalized sense. In particular, local well-posedness of \eqref{equ} has been established in  \cite{Hai14} and \cite{GIP15}.
These theories permit to treat a large number of singular subcritical SPDEs  including the Kardar-Parisi-Zhang (KPZ) equation and the stochastic quantization equations for quantum fields (see \cite{Hai13, CC18, GP17,  HS16} and references therein) and the 3D Navier--Stokes system  driven by space--time white noise \cite{ZZ15}.


On the other hand,  the question of {\it global} existence is even more challenging. We refer to \cite{CFW24} and \cite{HZZ23} for a
detailed discussion and further references.
Concerning \eqref{equ}, in cases where $F$ is linear, global solutions over time can be directly derived through fixed-point arguments (see, e.g., \cite{HL}). When $F$ has zeros that yield trivial super- and subsolutions, a priori estimates can be derived using comparison arguments (see, e.g., \cite{CFG17}). For general  $F\in C_b^2$ nonlinear term will appear in the solution theory of \eqref{equ} which makes it difficult to directly obtain global solutions to \eqref{equ}. Very recently, novel a priori estimates have been derived in \cite{CFW24}  within the framework of regularity structures, which leads to global well-posedness
for a certain range of $\kappa$ (see Remark~\ref{rem13} below).
 In this paper, using the idea in \cite{CFW24} that transferring the nonlinearity appeared in the original estimate to a transport term, we establish global well-posedness using paracontrolled calculus.

Let $\sI=\sL^{-1}$ and $T>0$.

\bd \label{enhanced}
We say that  $\eta\in L^\infty_TC^{-1-\kappa}$ with $0<\kappa<1/6$ can be lifted to an enhanced noise if there exist $\eta_n\in L^\infty_TC^\infty$
 such that
$\eta_n$ converges to $\eta$ in $L_T^\infty C^{-1-\kappa}$, and  there are constants $c_n$ and a function $g\in L_T^\infty C^{-2\kappa}$ such that
\begin{align*}
\lim_{n\to\infty}\|\sI (\eta_n)\circ\eta_n-c_n-g\|_{L_T^\infty C^{-2\kappa}}=0\;,
\end{align*}
where $\circ$ is paraproduct introduced in Section \ref{s:not}.

For notational convenience, we shall write
$$
g=:\sI (\eta)\circ\eta.
$$
We also call $(\eta,\sI(\eta)\circ \eta)$ the enhanced noise.
\ed

Let $\bar\kappa$ be the smaller solution to $4\kappa^2-7\kappa+1=0$, namely, $\bar\kappa=\frac18(7-\sqrt{33})\approx 0.15693$.

\bt\label{th:1} Let $\eta\in L^\infty_TC^{-1-\kappa}$ can  be lifted to an enhanced noise with   $0<\kappa<\bar\kappa$. For every given initial condition $u_0\in L^\infty$, there exists a unique global-in-time solution to \eqref{equ} in a paracontrolled sense (see Definition \ref{def:sol}).
\et

In \cite{CFW24} a mutiscale analysis and a new version of reconstruction  theorem are developed in the framework of regularity structures to derive a priori estimates.
In comparison to \cite{CFW24}, our primary techniques encompass the classical paracontrolled calculus established in \cite{GIP15}, the maximum principle and the key representation to express the leading error part as a transport type term (see Lemma \ref{Ru} below) as inspired by  \cite{CFW24}, as well as  a localization approach developed in \cite{GH18}.

\br\label{rem13}
(i) Comparing with  \cite[Theorem~1]{CFW24},
it is easy to see that $\bar\kappa$ here is slightly bigger than that in \cite{CFW24} (which is approximately $0.132123$). Hence,  we can allow $\eta$ to be  rougher w.r.t. space variable. Upon examining the proof in Section \ref{sec:proof}, it becomes apparent that the appropriate norms of the solution (e.g., $\|u\|_{\mS^{\alpha,\gamma}_T}$ in Section \ref{sec:proof}) depend polynomially on the suitable norm of the enhanced noise, similar to the findings in \cite{CFW24}. Consequently, by replacing $\Delta$ with $\Delta-m$ for $m>0$, uniformly bounded moments of the solutions over time, as demonstrated in \cite{CFW24}, can also be achieved.

(ii) Typical examples of $\eta$ are provided by spatial white noise $\zeta$ or $(-\Delta)^{s}\zeta$ with $0\leq s<\bar\kappa/2$ (see \cite{Hai14, GIP15}). In this case $c_n$ is time dependent and we could take it as a time-independent constant with an extra term which could be bounded by a constant multiply $t^{-\eps}$ for any $\eps>0$. Such a term could easily be controlled in the following and we ignore it for simplicity. Our results can also be readily extended to the nonlinear term expressed as $\sum_{i=1}^kF_i(u)\eta_i$ with $\eta_i$ as in Definition \ref{enhanced}.

(iii) One application of the a priori estimates 
 in \cite{CFW24} is establishing global-in-time bounds for the solutions to the dynamical sine-Gordon model introduced in \cite{HS16}, in the regime  $\beta^2\in (4\pi,(1+\bar\kappa)4\pi)$. Our results here necessitate $\eta\in L_T^\infty C^{-1-\kappa}$, but the renormalized terms $[\sin(\beta Z)]$ and $[\cos(\beta Z)]$ from the sine-Gordon model in \cite{HS16} do not satisfy this condition, where $Z$ solves the stochastic heat equation $\sL Z=\xi$ for space-time white noise $\xi$ on $\mathbb{R}^+\times \mathbb{T}^2$. Since the renormalized terms remain in a space-time distribution space $C^{-1-\kappa}_{t,x}$, we expect that our techniques can be applied to obtain global solutions for the dynamical sine-Gordon model using space-time paraproducts developed in \cite{BBF18}.
\er

In Section~\ref{s:not}, we summarize frequently used notations. In Section \ref{sec:loc} we recall the notion of paracontrolled solutions to \eqref{equ}.  We give the proof of Theorem \ref{th:1} in Section \ref{sec:proof}.  In Appendix, we collect several  results on paraproduct estimates and Schauder estimates.

{\bf Acknowledgments.}
H.S. gratefully acknowledges financial support from NSF grants
DMS-1954091 and CAREER DMS-2044415. R.Z. and X.Z. are grateful to
the financial supports   by National Key R\&D Program of China (No. 2022YFA1006300).
R.Z. gratefully acknowledges financial support from the NSFC (No. 12271030).  X.Z. is grateful to
the financial supports   by National Key R\&D Program of China (No. 2020YFA0712700) and  the NSFC (No.
12090014, 12288201)  and the support by key Lab of Random Complex Structures and Data Science, Youth Innovation Promotion Association (2020003), Chinese Academy of Science. This work is funded by the Deutsche Forschungsgemeinschaft (DFG, German Research Foundation) – Project-ID 317210226--SFB 1283.
\section{Preliminaries and notations}\label{s:not}
 Throughout the paper, we employ the notation $a\lesssim b$ if there exists a constant $c>0$ such that $a\leq cb$ and we write $a\simeq b$ if $a\lesssim b$ and $b\lesssim a$.
 We use $\cS'(\mT^d)$ to denote the space of tempered distributions on $\mT^d$.
 Let $\sF$ and $\sF^{-1}$ be the Fourier transform and the inverse of Fourier transform.
 For $j\geq -1$, let $\Delta_j$ be the usual block operator used in the Littlewood-Paley decomposition with $\Delta_j f=\sF^{-1}\rho_j\sF f$ with $(\rho_j)$ being a dyadic partition of unity. For index $i,j$ we write $i\sim j$ if $2^i\simeq 2^j$. We  introduce the following  Besov spaces (cf. \cite{BCD}):
\bd\label{Def25}
Let $p,q\in[1,\infty]$ and $\alpha\in\mR$. The Besov space $B^\alpha_{p,q}$ is defined by
$$
B^\alpha_{p,q}:=\left\{f\in\cS'(\mT^d):
\|f\|_{B^\alpha_{p,q}}:=\left(\sum_j 2^{\alpha jq}\|\Delta_j f\|_{L^p}^q\right)^{1/q}<\infty\right\}.
$$
The  H\"older-Zygmund space is defined by
$$
C^\alpha:=B^\alpha_{\infty,\infty}.
$$
\ed
In the following we used notation $\|\cdot\|_{\alpha}=\|\cdot\|_{C^\alpha}$ for simplicity.
It is well-known (c.f. \cite[page 99]{BCD}, \cite[Remark 2.6]{ZZZ20}) that   for any $0<\alpha\notin \mN$
\begin{equation}\label{f:sCalpha}
	\|f\|_{C^\alpha}\simeq \|f\|_{\sC^\alpha}\;,\end{equation}
where
$$ \|f\|_{\sC^\alpha}\eqdef \sum_{j=0}^{k}\|\nabla^j f\|_{L^\infty}+\sup_{x\neq y}\frac{|\nabla^kf(x)-\nabla^kf(y)|}{|x-y|^{\alpha-k}}\;,$$
for $k=\lfloor\alpha\rfloor$.

Given a Banach space $E$ with a norm  $\|\cdot\|_E$, we write  $C_TE$  to denote the space of continuous functions from $[0,T]$ to $E$.   For $p\in [1,\infty]$ we write $L^p_TE=L^p(0,T;E)$ for the space of $L^p$-integrable functions from $[0,T]$ to $E$, equipped with the usual $L^p$-norm.
For $\alpha\in(0,1)$ we  define $C^\alpha_TE$ as the space of $\alpha$-H\"{o}lder continuous functions from $[0,T]$  to $E$, endowed with the norm $$\|f\|_{C^\alpha_TE}\eqdef\sup_{s,t\in[0,T],s\neq t}\frac{\|f(s)-f(t)\|_E}{|t-s|^\alpha}+\sup_{t\in[0,T]}\|f(t)\|_{E}\;.$$
Set for $\gamma\in[0,1)$
$$\|f\|_{\mC^{\alpha,\gamma}_T}\eqdef\sup_{t\in[0,T]}t^\gamma\|f(t)\|_{\alpha}\;,\quad \|f\|_{C_T^{\alpha,\gamma}L^\infty}\eqdef\sup_{0\leq s< t\leq T}s^\gamma\frac{\|f(t)-f(s)\|_{L^\infty}}{|t-s|^{\alpha}}\;,$$
$$\|f\|_{\mS^{\alpha,\gamma}_T}\eqdef\|f\|_{\mC^{\alpha,\gamma}_T}+\|f\|_{C_T^{\alpha/2,\gamma}L^\infty}\;.$$
We also set
$$\mC^{\alpha,\gamma}_T\eqdef\{f:\|f\|_{\mC^{\alpha,\gamma}_T}<\infty\} \;,\qquad\mS^{\alpha,\gamma}_T\eqdef \{f:\|f\|_{\mS^{\alpha,\gamma}_T}<\infty.\}\;,$$
$\mS^{\alpha}_T\eqdef \mS^{\alpha,0}_T$
and $\|f\|_{\mL^\infty_T}=\|f\|_{L^\infty_TL^\infty}$. We also utilize $C_b^k$, where $k\in\mN$, to denote the spaces of bounded functions with bounded derivatives up to the $k$-th order and $\|f\|_{C_b^k}\eqdef \sum_{j=0}^k\|\nabla^jf\|_{L^\infty}$.

We recall the following paraproduct introduced by Bony (see \cite{Bon81}). In general, the product $fg$ of two distributions $f\in C^\alpha, g\in C^\beta$ is well defined if and only if $\alpha+\beta>0$. In terms of Littlewood-Paley blocks, the product $fg$ of two distributions $f$ and $g$ can be formally decomposed as
$$fg=f\prec g+f\circ g+f\succ g,$$
with $$f\prec g=g\succ f=\sum_{j\geq-1}\sum_{i<j-1}\Delta_if\Delta_jg, \quad f\circ g=\sum_{|i-j|\leq1}\Delta_if\Delta_jg.$$

See Appendix for useful paraproduct and related commutator estimates. Finally we introduce localizers in terms of Littlewood-Paley expansions. Let $R\in\mathbb{\mR}^+$. For $f\in\mathcal{S}'(\mathbb{T}^d)$ we define
$$\Delta_{>R}\eqdef\sum_{j> R}\Delta_j,\quad \Delta_{\leq R}\eqdef\sum_{j\leq R}\Delta_j.$$
Then it holds, in particular for $\alpha\leq \beta\leq \gamma$, that
\begin{align}\label{loc}\|\Delta_{>R}f\|_{\alpha}\lesssim 2^{-R(\beta-\alpha)}\|f\|_\beta,\quad \|\Delta_{\leq R}f\|_{\gamma}\lesssim 2^{R(\gamma-\beta)}\|f\|_\beta\,.
\end{align}
In the following we choose $\eps>0$ arbitrarily small constant.

\section{Paracontrolled Solution}\label{sec:loc}

In this section, we provide the definition of paracontrolled solutions. The multiplication of $F(u)$ and $\eta$ is not well-defined, as the expected sum of their regularities is not strictly positive for the resonant product $F(u) \circ \eta$ to be well-defined, as seen in Lemma~\ref{lem:para}. Collecting the terms that make $u$ too irregular leads to the following paracontrolled ansatz:
\begin{align}\label{anastz}
	u=F(u)\prec \sI(\eta)+u^\sharp.
\end{align}
where $\sI f(t)=\int_0^t P_{t-s}f(s)\dif s$ with $P_{t-s}=e^{(t-s)\Delta}$.
Then $u^\sharp$ becomes more regular than $u$ since
\begin{equation}\label{usharp}
	u^\sharp=u-\sI(F(u)\prec \eta)+[\sI,F(u)\prec]\eta,
\end{equation}
where $[\sI,F(u)\prec]\eta$ denotes the commutator between $\sI$ and $F(u)\prec$ given by
$$
[\sI,F(u)\prec]\eta=\sI(F(u)\prec \eta)-F(u)\prec \sI(\eta).
$$
We use the same notation for  other commutators as well.
The second  terms on the right hand side of \eqref{usharp}  cancel the irregular terms in $u$ whereas the last term has better regularity by Lemma \ref{commutator1}.
 Since using a simple paralinearization lemma from \cite[Lemma 2.6]{GIP15}, $F(u)\in C^{1-\kappa}$ and $F(u)-F'(u)\prec u$ has better regularity to make $(F(u)-F'(u)\prec u)\circ \eta$ well-defined. This leads to the following decomposition of $F(u)\circ \eta$:
 $$F(u)\circ \eta=   (F(u)-F'(u)\prec u)\circ \eta+(F'(u)\prec u)\circ \eta\;.$$
We then use paracontrolled ansatz \eqref{anastz} to replace $u$ in the second $F'(u)\prec u$ and  the commutator Lemma  \ref{lem:com2}, we write formally $F(u)\circ \eta$ as
\begin{equation}\label{product}
	\begin{aligned}
		F(u)\circ \eta&=(F(u)-F'(u)\prec u)\circ \eta+(F'(u)\prec u^\sharp)\circ  \eta+\text{com}(F'(u),F(u)\prec \sI(\eta),\eta)\\&+F'(u) \text{com}(F(u), \sI(\eta), \eta)+F(u)F'(u)(\sI(\eta)\circ \eta)\;,
	\end{aligned}
\end{equation}
where $\text{com}$ is the  commutator and is given by $\text{com}(f,g,h)=(f\prec g)\circ h-f(g\circ h)$ for $f,g,h$ smooth introduced in Lemma \ref{lem:com2}. Based on the paraproduct estimates and commutator estimates, we find every term in \eqref{product} is well-defined except $\sI(\eta)\circ\eta$.
As $\eta$ can be lifted to  enhanced noise (see Definition \ref{enhanced}), $\sI(\eta)\circ \eta$  could be understood in a  renormalized sense (see \cite{Hai14, GIP15}). 

Hence the rigorous paracontrolled ansatz and the definition of  $F(u)\circ \eta$ read as \eqref{product}.


\begin{definition}\label{def:sol}
	We say that a pair of $(u,u^\sharp) \in \mS_T^{\alpha,\gamma}\times \mC_T^{2\alpha,\gamma}$ with some $0\leq\gamma<1$, $2/3<\alpha<1$ is a paracontrolled solutions to \eqref{equ}
	provided \eqref{anastz} holds
	and the equation \eqref{equ} holds in the analytically weak sense with the resonant product $F(u) \circ \eta$  given by \eqref{product}.
\end{definition}

We have the following result from \cite[Theorem 5.4]{GIP15}.
\begin{theorem}\label{local}
 For any $u_0\in L^\infty$ there exists a $T^*>0$ and a unique maximal paracontrolled solution $(u,u^\sharp)$ to \eqref{equ} on $[0,T^*)$ with $u(0)=u_0$.
\end{theorem}

\br In \cite[Theorem 5.4]{GIP15} the initial condition is in $C^{2/3+}$, where and in the sequel  $C^{\alpha+}$ denote the space $C^{\alpha+\eps}$ for arbitrarily small $\eps>0$. Since $\eps$ can be neglected, we omit it here.  By considering space with time singularity we could deduce the above result for initial condition in $L^\infty$ (see \cite{ZZ15}).
\er

\section{Proof of Theorem \ref{th:1}}\label{sec:proof}
In this section, we present the proof of the main results. Initially, we decompose the equation \eqref{equ} appropriately and establish a crucial lemma in Section \ref{sec:dec}. This lemma represents the leading error part as a transport term. Subsequently, in the remainder of this section, we prove several a priori estimates for the involved quantities in various spaces.

 When the initial condition $u_0\in L^\infty$, we are aware that the solution will reside in $C^{1-\kappa}$ after a short time, allowing us to restart the solution from $C^{1-\kappa}$. Subsequently, we address the case that $u_0\in C^{1-\kappa}$. By Theorem \ref{local} it suffices to deduce a priori bound for $(u,u^\sharp)$ in $\mS_T^{\alpha,\gamma}\times \mC_T^{2\alpha,\gamma}$ with $\alpha>2/3,\gamma\in[0,1)$.

  The paracontrolled anastz is as follows:
\begin{align}\label{para}u=F(u)\prec \sI(\Delta_{>R}\eta)+u^\sharp:=u_1+u^\sharp.
\end{align}
where  $R$ is a large constant which will be given later. Here $u^\sharp$ is a little different from the one in \eqref{anastz}. As $\sI(\Delta_{\leq R}\eta)$ is smooth, it is easy to see a priori bound for $(u,u^\sharp)$ here implies a priori bound for $(u,u^\sharp)$ in \eqref{anastz}.
With a minor abuse of notation we still write $u^\sharp$ here.
 We also have by Lemma \ref{lem:para} and Lemma \ref{Le11}
\begin{align}\label{u1}\|u_1\|_{C_TC^{1-\kappa}}\lesssim 1.
\end{align}

\subsection{Decomposition}\label{sec:dec}
  For $u^\sharp$ we use paraproduct decomposition  to have
 \begin{align*}
 	\sL u^\sharp&= -[\sL(F(u)\prec \sI(\Delta_{>R}\eta ))-F(u)\prec \Delta_{>R}\eta]
 \\&\quad+F(u)\succ \Delta_{>R}\eta +F(u)\circ \Delta_{>R}\eta+F(u)\Delta_{\leq R}\eta\;,
\\ u^\sharp(0)&=u_0\in C^{1-\kappa}\;.
 \end{align*}
We then define $R(u)\eqdef F(u)-F'(u)\prec u$ and
substitute $F(u)=R(u)+F'(u)\prec u$ into $F(u)\circ \Delta_{>R}\eta$ to have
\begin{equation}\label{eq:usharp}
	\aligned
 \sL u^\sharp&=-[\sL(F(u)\prec \sI(\Delta_{>R}\eta ))-F(u)\prec \Delta_{>R}\eta]
  \\&\quad+F(u)\succ \Delta_{>R}\eta +R(u)\circ \Delta_{>R}\eta +(F'(u)\prec u)\circ  \Delta_{>R}\eta
  +F(u)\Delta_{\leq R}\eta\;,
  \\ u^\sharp(0)&=u_0\in C^{1-\kappa}\;.
  \endaligned
 \end{equation}

By the paraproduct estimates in Lemma \ref{lem:para}, the $L^\infty$-norm of every term on the RHS of \eqref{eq:usharp}, except $R(u)\circ \Delta_{>R}\eta$, depends linearly on $u$ and $u^\sharp$. According to Lemma \ref{Ru} below, we understand that the leading term of $R(u)\circ \Delta_{>R}\eta$ is represented by a transport-type term $\nabla u^{\sharp}\cdot (G(u)\circ \eta)$. To apply the maximum principle, which necessitates the majority of terms to remain in $L^\infty$, we decompose the equation \eqref{eq:usharp} into two equations with $u^\sharp=u_1^\sharp+u_2^\sharp$: Formally, the right-hand side of $u_2^\sharp$ is in $L^\infty$.
 \begin{align}
 	\sL u_1^\sharp&= -[\sL(F(u)\prec \sI(\Delta_{>R}\eta ))-F(u)\prec \Delta_{>R}\eta]\no
  \\&\quad+F(u)\succ \Delta_{>R}\eta+F(u)F'(u)\Delta_{>R_1}(\sI(\eta)\circ \eta)\;, \label{eq:u1sharp}
  \\ u_1^\sharp(0)&=u_0\in C^{1-\kappa}\;.\no
  \\
\sL u_2^\sharp&= R(u)\circ \Delta_{>R}\eta +F(u)\Delta_{\leq R}\eta\no
\\&\quad+(F'(u)\prec u)\circ  \Delta_{>R}\eta-F(u)F'(u)\Delta_{>R_1}(\sI(\eta)\circ \eta)\;,\label{eq:u2sharp}
\\ u_2^\sharp(0)&=0\;.\no
 \end{align}
Here the parameters $R$ and $R_1$ will be chosen below and the last term $F(u)F'(u)\Delta_{>R_1}(\sI(\eta)\circ \eta)$ comes from the paraproduct decomposition of $(F'(u)\prec u)\circ \Delta_{>R}\eta$ as in \eqref{product}. We have to put $\Delta_{>R_1}(\sI(\eta)\circ \eta)$ into equation of $u_1^\sharp$ since $\sI(\eta)\circ \eta \in C^{-2\kappa}$ not in $L^\infty$.

The primary challenge lies in obtaining a suitable estimate for $R(u)\circ \eta$.
In \cite{GIP15}, it has been demonstrated that $\|R(u)\|_{1+}$ exhibits superlinear dependence on $u$. In the subsequent lemma, we provide a more nuanced estimate. Specifically, we decompose $R(u)\circ \eta$ into a transport-type part and other components that depend linearly on appropriate norms of $u_1^\sharp$, $u$, and $u_2^\sharp$.

\begin{lemma}\label{Ru} Let $\eta\in C^{\alpha-2}$ with $2/3<\alpha<1$. Then there exists  a locally bounded map $G:C^{1-\kappa}\to  C^{1}$ such that for $\eps>0$
$$\|R(u)\circ \eta-\nabla u^{\sharp}(G(u)\circ \eta) \|_{L^\infty}\lesssim\|\eta\|_{\alpha-2}(1+ \|u_1^\sharp\|_{2\alpha}+\|u\|_{\alpha})+\|\eta\|_{{-2+\eps}} \|u_2^\sharp\|_{{2-\eps/2}}\;,$$
and $G(u)=G^1+G^2$. For $\delta\in(0,1]$
 $$\|G^1\|_{2-\kappa}\lesssim 1+\|u_1^\sharp\|_{1-\kappa}\;,\quad\|G^2\|_{1+\delta}\lesssim \|u_2^\sharp\|^\delta_{1+}\;.$$
\end{lemma}
\begin{proof}
Similarly as \cite[Lemma 2.6]{GIP15}  (``paralinearization lemma'') we  write
$$R(u)=\sum_{i\geq-1}(\Delta_i F(u)-S_{i-1}F'(u)\Delta_iu)=\sum_{i\geq-1} v_i\;.$$
In comparison to \cite[Lemma 2.6]{GIP15}, we will perform additional decomposition of each $v_i$ involving the decomposition of $u=u_1+u^\sharp=u_1+u_1^\sharp+u_2^\sharp$. This allows us to achieve more refined bounds.
More precisely, we write the Littlewood-Paley projections as convolutions and obtain
\begin{align*}
v_i&=\int K_i(x-y)K_{<i-1} (x-z) \Big(F(u(y))-F(u(z))-F'(u(z))(u(y)-u(z))\Big)\;\dif y\dif z
\\&=\int K_i(x-y)K_{<i-1} (x-z)\bigg[\int_0^1F'(u(z)+\tau(u(y)-u(z)))\dif\tau-F'(u(z))\bigg]
\\&\quad\times(u(y)-u(z))\;\dif y\dif z
\\&=I^i_1+I^i_2\;,
\end{align*}
where $K_i=\sF^{-1}\rho_i, K_{<i-1}=\sum_{j<i-1}K_j$ and
\begin{align*}
	I^i_1\eqdef& \int K_i(x-y)K_{<i-1} (x-z)\bigg[\int_0^1F'(u(z)+\tau(u(y)-u(z)))\;
	\dif\tau-F'(u(z))\bigg]
	\\&\times(u^\sharp(y)-u^\sharp(z))\;\dif y\dif z\;,
\\I^i_2\eqdef &\int K_i(x-y)K_{<i-1} (x-z) \int_0^1\int_0^1F''(u(z)+\tau r(u(y)-u(z)))\dif\tau \dif r
\\&\times(u(y)-u(z))(u_1(y)-u_1(z))\;\dif y\dif z\;.\end{align*}
In the following we calculate $\sum_{i\geq -1}I^i_2\circ \eta$ and $\sum_{i\geq -1}I^i_1\circ \eta$ separately.
We will find that $\sum_{i\geq -1}I^i_2\circ \eta$ is part of remainder term.  Using \eqref{u1}, \eqref{f:sCalpha} and $F\in C_b^2$ we obtain
\begin{align*}
	\|I^i_2\|_{L^\infty} &\lesssim \|F\|_{C_b^2} \|u\|_\alpha\|u_1\|_{\alpha}\int| K_i(x-y)K_{<i-1} (x-z)|\times |y-z|^{2\alpha}\;\dif y\dif z
	\\
	 &\lesssim 2^{-2\alpha i}\|u\|_{\alpha}\|u_1\|_{\alpha}\lesssim 2^{-2\alpha i}\|u\|_{\alpha}\;,\end{align*}
 where we use $K_i(x)=2^{id}K(2^ix), K_{<i-1}(x)=2^{id}\bar K(2^ix)$  and change of variable in the last step. Since the Fourier transform of
 each $I^i_{2}$ is supported in $2^i\cA$ with $\cA$ being an annulus, we obtain
 $$\Big\|\sum_iI^i_2\Big\|_{2\alpha}\lesssim \|u\|_{\alpha}\;.$$
Thus using paraproduct estimates Lemma \ref{lem:para} we deduce
$$\Big\|\sum_iI^i_2\circ \eta\Big\|_{3\alpha-2}\lesssim \|u\|_{\alpha}\|\eta\|_{\alpha-2}\;.$$

Write $I^i_1=I^i_{111}+I^i_{112}+I^i_{12}$
where
\begin{equation*}
	\aligned
	I^i_{111}\eqdef &	\int K_i(x-y)K_{<i-1} (x-z)\Big[\int_0^1F'(u(z)+\tau(u(y)-u(z)))\dif\tau-F'(u(z))\Big]\\&\times \Big(u_1^\sharp(y)-u_1^\sharp(z)-\nabla u_1^\sharp(z)\cdot (y-z)\Big)\;\dif y\dif z\;,
\\I^i_{112}\eqdef &	\int K_i(x-y)K_{<i-1} (x-z)\Big[\int_0^1F'(u(z)+\tau(u(y)-u(z)))\dif\tau-F'(u(z))\Big]\\&\times \Big(u_2^\sharp(y)-u_2^\sharp(z)-\nabla u_2^\sharp(z)\cdot (y-z)\Big)\;\dif y\dif z\;,
	\\
I^i_{12}\eqdef 	&	\int K_i(x-y)K_{<i-1} (x-z)\Big [\int_0^1F'(u(z)+\tau(u(y)-u(z)))\dif\tau-F'(u(z))\Big]
\\&\times (y-z)\cdot\nabla u^\sharp (z)\;\dif y\dif z\;.
\endaligned
\end{equation*}
We will also find that
$\sum_iI^i_{111}\circ \eta$ and $\sum_iI^i_{112}\circ \eta$ are part of remainder terms.  Using $F\in C_b^1$ and \eqref{f:sCalpha},
\begin{align*}
	\|I^i_{111}\|_{L^\infty} &\lesssim \|F\|_{C_b^1} \|u_1^\sharp\|_{2\alpha}\int| K_i(x-y)K_{<i-1} (x-z)|\times |y-z|^{2\alpha}\;\dif y\dif z
	\\
	&\lesssim 2^{-2\alpha i}\|u_1^\sharp\|_{2\alpha}\;.
\end{align*}
 Since the Fourier transform of
 each $I^i_{11}$ is supported in $2^i\cA$ with $\cA$ being an annulus, we obtain
 $$\Big\|\sum_iI^i_{111}\Big\|_{2\alpha}\lesssim \|u_1^\sharp\|_{2\alpha}\;.$$
Thus using paraproduct estimates Lemma \ref{lem:para} we obtain
$$\Big\|\sum_iI^i_{111}\circ \eta\Big\|_{3\alpha-2}\lesssim \|u^\sharp_1\|_{2\alpha}\|\eta\|_{\alpha-2}\;.$$
Similarly we use \eqref{f:sCalpha} to obtain for $\eps>0$
$$\Big\|\sum_iI^i_{112}\circ \eta\Big\|_{C^+}\lesssim \|u_2^\sharp\|_{\sC^{2-\eps/2}}\|\eta\|_{-2+\eps}\lesssim \|u_2^\sharp\|_{{2-\eps/2}}\|\eta\|_{-2+\eps}\;.$$
In the following we focus on $I^i_{12}$ and write it as
\begin{align*}
	 \Big(\sum_i I^i_{12}\Big)\circ \eta= \sum_{k\sim l}\sum_{i\sim k} \Delta_kI^i_{12} \cdot  \Delta_l \eta\;.
\end{align*}
We use the definition of $I_{12}^i$ and write $\nabla u^\sharp (z)=\nabla u^\sharp (z)-\nabla u^\sharp (v)+\nabla u^\sharp (v)$ to have
\begin{align*}
	&(\Delta_kI^i_{12}\cdot \Delta_l \eta)(v)=J_1+J_2\;,
\end{align*}
where
\begin{align*}
J_1&\eqdef\int K_k(v-x)\int K_i(x-y)K_{<i-1} (x-z)
\Big[\int_0^1F'(u(z)+\tau(u(y)-u(z)))\dif\tau-F'(u(z))\Big]
\\&\quad\times (y-z)(\nabla u^\sharp (z)-\nabla u^\sharp (v))\:\dif y\dif z\dif x\; \cdot \Delta_l\eta (v)\;,
\\
J_2&\eqdef\nabla u^\sharp(v) \Delta_k \Big(\int K_i(\cdot-y)K_{<i-1} (\cdot-z)\Big[\int_0^1F'(u(z)+\tau(u(y)-u(z)))\dif\tau-F'(u(z))\Big]
\\&\quad\times (y-z)\;\dif y\dif z\Big)(v) \cdot \Delta_l\eta (v)\;.
\end{align*}
Concerning $J_1$ we use change of variable and \eqref{f:sCalpha} to  have for $\eps>0$
\begin{align*}
|J_1| &\lesssim\int \int |K_k(v-x)K_i(x-y)K_{<i-1} (x-z)|\cdot |y-z|
\\
&\quad\times(\|u_1^\sharp\|_{2\alpha}|z-v|^{2\alpha-1}+\|u_2^\sharp\|_{\sC^{2-\eps/2}}|z-v|^{1-\eps/2})\;\dif y\dif z\dif x\cdot\|\Delta_l\eta\|_{L^\infty}
\\
&\lesssim\int \int |K_k(x)K_i(x-y)K_{<i-1} (x-z)| \cdot|y-z|
\\
&\quad\times (\|u_1^\sharp\|_{2\alpha}|z|^{2\alpha-1}+\|u_2^\sharp\|_{2-\eps/2}|z|^{1-\eps/2})\;\dif y\dif z\dif x\cdot\|\Delta_l\eta\|_{L^\infty}
\\
&\lesssim  2^{-2i\alpha}2^{(2-\alpha) l}\|u_1^\sharp\|_{2\alpha}\|\eta\|_{\alpha-2}+2^{-i(2-\eps/2)}2^{l(2-\eps)}\|u_2^\sharp\|_{2-\eps/2}\|\eta\|_{-2+\eps}\;,
\end{align*}
 where we use $k\sim i$ and $K_i(x)=2^{id}K(2^ix), K_{<i-1}(x)=2^{id}\bar K(2^ix)$  and change of variable in the last step.
Taking sum over $k\sim l\sim i$ we obtain
 $$\Big\|\sum_{k\sim l\sim i}J_1\Big\|_{L^\infty}
 \lesssim \|u_1^\sharp\|_{2\alpha}\|\eta\|_{\alpha-2}+\|u_2^\sharp\|_{2-\eps/2}\|\eta\|_{-2+\eps}\;.$$
Concerning $J_2$ we define
\begin{align*}
G^1\eqdef \sum_iG^1_i
&\eqdef \sum_i  \int K_i(\cdot-y)K_{<i-1} (\cdot-z)
\\&\cdot\int_0^1\Big[F'(u(z)+\tau(u(y)-u(z)))-F'(u(z)+\tau(u^\sharp_2(y)-u^\sharp_2(z)))\Big]\dif\tau\cdot(y-z)\dif y\dif z\;,
\\
G^2\eqdef \sum_iG^2_i
&\eqdef \sum_i  \int K_i(\cdot-y)K_{<i-1} (\cdot-z)
\\&\cdot\Big[\int_0^1F'(u(z)+\tau(u^\sharp_2(y)-u^\sharp_2(z)))\dif\tau-F'(u(z))\Big](y-z)\dif y\dif z\;.
\end{align*}
Hence, $J_2$ is written as $\nabla u^\sharp \cdot \Delta_k (G^1_i+G_i^2)\Delta_l\eta$. Taking sum over $k\sim l\sim i$ we obtain that
$$\sum_{k\sim l\sim i}J_2=\nabla u^\sharp\cdot G\circ \eta\;.$$
By direct calculation as above we obtain for $u^1=u_1+u_1^\sharp$
\begin{align*}
	\|G^1_i\|_{L^\infty}
	&\lesssim \|F\|_{C_b^2}\int |K_i(\cdot-y)K_{<i-1} (\cdot-z)|\cdot|u^1(y)-u^1(z)||y-z|\dif y\dif z
	\\
	&\lesssim \|u^1\|_{1-\kappa}\int |K_i(\cdot-y)K_{<i-1} (\cdot-z)|\cdot|y-z|^{2-\kappa}\dif y\dif z
	\lesssim  \|u^1\|_{1-\kappa}2^{(-2+\kappa)i}\;.
\end{align*}
Since the Fourier transform of  $G_i^1$ is supported in $2^i\mathcal{A}$ with $\mathcal{A}$ being an annulus, the bound for $G^1$ follows from \eqref{u1}. Similarly we obtain
\begin{align*}
	\|G^2_i\|_{L^\infty} &\lesssim \|F\|_{C_b^2}\int |K_i(\cdot-y)K_{<i-1} (\cdot-z)|\cdot|u^\sharp_2(y)-u^\sharp_2(z)|^{\delta}|y-z|\dif y\dif z
	\\&\lesssim \|u^\sharp_2\|^\delta_{\sC^1}\int |K_i(\cdot-y)K_{<i-1} (\cdot-z)|\cdot|y-z|^{1+\delta}\dif y\dif z
	\lesssim \|u^\sharp_2\|^\delta_{1+}2^{(-1-\delta)i}\;.\end{align*}
Thus the final result holds by similar argument as for $G^1$.
\end{proof}

  \subsection{Bounds for  $u_1^\sharp$ in $C_TC^{1-\kappa}$ and $\mC_T^{2\alpha,\gamma}$} In the following we choose $\gamma=\frac{2\alpha+\kappa-1}2$ and $\alpha$ close to $2/3$.
   By \eqref{eq:u1sharp} we have
 \begin{equation}\label{u1sharp}
 	\aligned
 	u_1^\sharp &= -F(u)\prec \sI(\Delta_{>R}\eta )+P_tu_0\\
 &\quad +\sI\Big(F(u)\prec \Delta_{>R}\eta+F(u)\succ \Delta_{>R}\eta+F(u)F'(u)\Delta_{>R_1}(\sI(\eta)\circ \eta)\Big)\;.
 \endaligned
 \end{equation}
For the first term on the RHS of \eqref{u1sharp} by Lemma \ref{lem:para} we have
 \begin{align*}\| F(u)\prec \sI(\Delta_{>R}\eta) \|_{1-\kappa}\lesssim \|F\|_{L^\infty}\|\Delta_{>R}\eta\|_{-1-\kappa}\lesssim1\;.\end{align*}
 For the second line of  \eqref{u1sharp} we use  paraproduct estimate Lemma \ref{lem:para} to obtain
  \begin{equation}
  	\label{bd}
  	\aligned
  	\|F(u)\prec \Delta_{>R}\eta\|_{-1-\kappa}\lesssim\|F\|_{L^\infty}\|\Delta_{>R}\eta\|_{-1-\kappa} \lesssim1\;,
  	\endaligned
  \end{equation}
  \begin{equation}
  	\label{bd1}
  	\aligned
  	\|F(u)\succ\Delta_{>R}\eta\|_{-1-\kappa}\lesssim\|\Delta_{>R}\eta\|_{-1-\kappa}\lesssim 1\;,
  	\endaligned
  \end{equation}
  and use \eqref{loc} to have for $\eps>0$ small enough 
  \begin{equation}\label{bd2}
  	\aligned
  	&\|F(u)F'(u)\Delta_{>R_1}(\sI(\eta)\circ\eta)\|_{-\alpha+\eps}
  	\lesssim \|F(u)F'(u)\|_{\alpha}\|\Delta_{>R_1}(\sI(\eta)\circ\eta)\|_{-\alpha+\eps}
  	\\ & \lesssim  \|u\|_\alpha 2^{R_1(-\alpha+2\kappa+\eps)}\|\sI(\eta)\circ\eta\|_{-2\kappa}\lesssim \|u\|_\alpha 2^{R_1(-\alpha+2\kappa+\eps)}\;.
  	\endaligned
  \end{equation}
Combining these estimates and using Lemma \ref{Le11} and Lemma \ref{Le12} we arrive at
\begin{align*}\|u_1^\sharp\|_{C_TC^{1-\kappa}}\lesssim1+\|u\|_{\mS_T^{\alpha,\gamma\alpha/2+\eps}}2^{R_1(-\alpha+2\kappa+\eps)},\end{align*}
where we used $\alpha$ close to $2/3$.
Now we choose
 \begin{align}\label{c:R1}
	\|u\|_{\mS_T^{\alpha,\gamma\alpha/2+\eps}}2^{R_1(-\alpha+2\kappa+\eps)}\simeq1,\end{align}
 which implies that
\begin{align}\label{bdu1}\|u_1^\sharp\|_{C_TC^{1-\kappa}}\lesssim1.\end{align}

Furthermore, we derive a bound for $u_1^\sharp$ in a better regularity space with time singularity at $t=0$. First we use  commutator estimate Lemma \ref{commutator1} to obtain
  $$\|(F(u)\prec \sI(\Delta_{>R}\eta ))-\sI(F(u)\prec \Delta_{>R}\eta)\|_{\mC_T^{2\alpha,\gamma}}\lesssim\|u\|_{\mS_T^{\alpha,\gamma}}\|\Delta_{>R}\eta\|_{L^\infty_TC^{\alpha-2}} ,$$
  and use paraproduct estimates Lemma \ref{lem:para} to obtain
  \begin{align*}
  	&\|F(u)\succ \Delta_{>R}\eta\|_{2\alpha-2}
  	\lesssim \|u\|_{\alpha}\|\Delta_{>R}\eta\|_{\alpha-2}\;,
  \end{align*}
Combining these estimates and \eqref{bd2} and using Lemma \ref{Le11} and Lemma \ref{lem:2.8} we arrive at
 \begin{align}\label{bd:u1sharp1}
 \|u_1^\sharp\|_{\mC_T^{2\alpha,\gamma}}&\lesssim\|u\|_{\mS_T^{\alpha,\gamma}}\|\Delta_{>R}\eta\|_{L^\infty_TC^{\alpha-2}}+\|P_tu_0\|_{\mC_T^{2\alpha,\gamma}}+\|u\|_{\mS_T^{\alpha,\gamma\alpha/2+\eps}}2^{R_1(-\alpha+2\kappa+\eps)}
 \\&\lesssim\|u\|_{\mS_T^{\alpha,\gamma}}2^{R(\alpha-1+\kappa)}\|\eta\|_{L^\infty_TC^{-1-\kappa}}+1.\no
\end{align}
Here in the last step we used \eqref{loc} and \eqref{c:R1}.
\subsection{Time regularity of $u$} In this section, we establish bounds on the time regularity of the function $u$. The idea is to constrain it within the temporal regularity of $u^\sharp_2$.
First we use Lemma \ref{lem:para} to have
\begin{align}\label{fu}
	\|F(u)\prec \sI(\Delta_{>R}\eta)\|_{C_T^{\alpha/2,\gamma}L^\infty}\lesssim \|u\|_{C_T^{\alpha/2,\gamma}L^\infty}\|\Delta_{>R}\eta\|_{L_T^\infty C^{\alpha-2}}+1\;.
\end{align}
Thus
 we have
 \begin{align*}
 \|u\|_{C_T^{\alpha/2,\gamma}L^\infty}
& \lesssim \|u\|_{C_T^{\alpha/2,\gamma}L^\infty}\|\Delta_{>R}\eta\|_{L^\infty_TC^{\alpha-2}}+\|u^\sharp\|_{C_T^{\alpha/2,\gamma}L^\infty}+1
 \\
 &\lesssim \|u\|_{C_T^{\alpha/2,\gamma}L^\infty}2^{R(\alpha-1+\kappa)}\|\eta\|_{L^\infty_TC^{-1-\kappa}}+\|u^\sharp\|_{C_T^{\alpha/2,\gamma}L^\infty}+1\;,
 \end{align*}
where we use localizer estimate \eqref{loc} in the second inequality.
 Choose $2^{R(\alpha-1+\kappa)}$ small enough such that the coefficient before $\|u\|_{C_T^{\alpha/2,\gamma}L^\infty}$ on the right hand side is smaller than $1/2$ and $\|u\|_{C_T^{\alpha/2,\gamma}L^\infty}$ could be absorbed by the right hand side. Thus we obtain
 \begin{align}\label{bdu}
 	\|u\|_{C_T^{\alpha/2,\gamma}L^\infty}
& \lesssim \|u^\sharp\|_{C_T^{\alpha/2,\gamma}L^\infty}+1\;.
 \end{align}
 Moreover, concerning $\|u_1^\sharp\|_{C_T^{\alpha/2,\gamma}L^\infty}$,
we use \eqref{u1sharp}. 
Using Lemma \ref{lem:2.8} we have
$$\|P_t u_0-P_su_0\|_{L^\infty}=\|(P_{t-s}-I)P_su_0\|_{L^\infty}\lesssim (t-s)^{\alpha/2}\|u_0\|_{\alpha}\;,$$
which implies that
\begin{align*}
	\|P_tu_0\|_{C_T^{\alpha/2,\gamma}L^\infty}\lesssim 1\;.
\end{align*}
Also, using Lemma \ref{Le11} and \eqref{bd} \eqref{bd1} \eqref{bd2},
the $C_T^{\alpha/2,\gamma}L^\infty$-norm for the second line of \eqref{u1sharp} can  be bounded by a constant,
   which by \eqref{fu} leads to
 \begin{align*}
 \|u_1^\sharp\|_{C_T^{\alpha/2,\gamma}L^\infty}
 &\lesssim \|u\|_{C_T^{\alpha/2,\gamma}L^\infty}\|\Delta_{>R}\eta\|_{L^\infty_TC^{\alpha-2}}+1
 \\
 &\lesssim\|u_1^\sharp\|_{C_T^{\alpha/2,\gamma}L^\infty}2^{R(\alpha-1+\kappa)}\|\eta\|_{L^\infty_TC^{-1-\kappa}}+\|u_2^\sharp\|_{C_T^{\alpha/2,\gamma}L^\infty}\|\eta\|_{L^\infty_TC^{-1-\kappa}}+1\;.\no
\end{align*}
Here in the last step we used  \eqref{bdu} and localizer estimate \eqref{loc}.
Choose $2^{R(\alpha-1+\kappa)}$ small enough such that the coefficient before $\|u_1^\sharp\|_{C_T^{\alpha/2,\gamma}L^\infty}$ on the right hand side is smaller than $1/2$ and $\|u_1^\sharp\|_{C_T^{\alpha/2,\gamma}L^\infty}$  could be absorbed by the right hand side. Thus we obtain
 \begin{align}\label{bd:u1sharp2}&\|u_1^\sharp\|_{C_T^{\alpha/2,\gamma}L^\infty}\lesssim\|u_2^\sharp\|_{C_T^{\alpha,\gamma}L^\infty}\|\eta\|_{L^\infty_TC^{-1-\kappa}}+1.
 \end{align}
 Thus combining with \eqref{u1} \eqref{bdu} and \eqref{bdu1} we arrive at
 \begin{align}\label{bd:utime}&
 	\|u\|_{\mS_T^{\alpha,\gamma}}=\|u\|_{\mC_T^{\alpha,\gamma}}+\|u\|_{C^{\alpha/2,\gamma}_TL^\infty}\lesssim 1+\|u_2^\sharp\|_{\mS_T^{\alpha,\gamma}}+\|u^\sharp\|_{C^{\alpha/2,\gamma}_TL^\infty}\lesssim\|u_2^\sharp\|_{\mS_T^{\alpha,\gamma}}+1\;,
 \end{align}
where we used \eqref{bd:u1sharp2} in the last step.
 Similar argument also implies
 \begin{align}\label{bd:utime1}&\|u\|_{\mS_T^{\alpha,\gamma\alpha/2+\eps}}\lesssim\|u_2^\sharp\|_{\mS_T^{\alpha,\gamma\alpha/2+\eps}}+1\;.
 \end{align}

  \subsection{Bound for $u_2^\sharp$ in $\mL^\infty_T$ and $\mS^{2-\eps/2,\gamma}_T$}

  We first prove the following maximum princple.
  \bt(Maximum principle)\label{Thmax}
Let $w$ be a strong solution to the following equation
\begin{align*}\partial_t w-\Delta w&=b\cdot \nabla w+f\;,
\\w(0)&=w_0,
\end{align*}
 with  $w_0\in L^\infty$ and $b\in \mL^\infty_T$. Then it holds that for $\gamma\in[0,1)$
\begin{align}\label{Max}
\|w\|_{\mL^\infty_T}\leq\|w_0\|_{L^\infty}+C(\gamma)T^{1-\gamma}\sup_{0\leq t\leq T}t^\gamma\|f(t)\|_{L^\infty}\;,
\end{align}
where $C(\gamma)$ is independent of $b$.
\et

\begin{proof}
We apply a probabilistic method, with reversed time. For a space-time function $f$, we set
$$
f^T(t,x):=f(T-t,x)\;.
$$
It is easy to see that $w^T(t,x)=w(T-t,x)$ solves the following backward equation:
\begin{align}\label{GG4}
\p_t w^T+\Delta w^T+b^T\cdot\nabla w^T+f^T=0\;,
\end{align}
with  the final condition
\begin{align}\label{GG5}
w^T(T,x)=w(0,x)=w_0(x)\;.
\end{align}
Under  $b\in \mL^\infty_T$, for each $(t,x)\in[0,T]\times\mR^d$,
it is well known that the following SDE has a (probabilistically) weak solution starting from $x$ at time $t$
$$
X^{t,x}_s=x+\sqrt{2}(W_s-W_t)+\int^s_tb^T(r,X^{t,x}_r)\dif r\;,\ \ \forall s\in[t,T]\;,
$$
where $W$ is a $d$-dimensional Brownian motion on some stochastic basis $(\Omega',\mathcal{F}',\mathbb{P})$.
Thus, for each fixed $(t,x)$, by  It\^o's formula, we have
\begin{align*}
\dif_s w^T(s,X^{t,x}_s)&=(\p_sw^T+\Delta w^T+b^T\cdot\nabla w^T)(s,X^{t,x}_s)\dif s\\
&\quad+(\sqrt{2}\nabla w^T)(s,X^{t,x}_s)\dif W_s\;,
\end{align*}
and by \eqref{GG4} and \eqref{GG5},
\begin{align*}
&\mathbb{E} w_0(X^{t,x}_{T})=\mathbb{E}w^T(T,X^{t,x}_{T})=w^T(t,x)-\mathbb{E}\int^{T}_tf^T(s,X^{t,x}_s)\dif s\;,
\end{align*}
which implies that
\begin{align*}
\|w^T\|_{\mL^\infty_T}
&\leq \|w_0\|_{L^\infty}+\int_0^T\|f(T-s)\|_{L^\infty}\dif s
\\&\leq\|w_0\|_{L^\infty}+\int_0^T(T-s)^{-\gamma}\dif s \sup_{0\leq t\leq T}t^\gamma\|f(t)\|_{L^\infty}
\\&\leq\|w_0\|_{L^\infty}+C(\gamma)T^{1-\gamma} \sup_{0\leq t\leq T}t^\gamma\|f(t)\|_{L^\infty}\;.
\end{align*}
Thus the result follows.
\end{proof}

 For equation \eqref{eq:u2sharp} we write $R(u)\circ \Delta_{>R}\eta$ on the RHS as
 $$R(u)\circ \Delta_{>R}\eta=(G(u)\circ \Delta_{>R}\eta)\cdot (\nabla u^\sharp_1+\nabla u^\sharp_2)+\cJ\;,$$
 where
 $$\cJ\eqdef R(u)\circ \Delta_{>R}\eta -(G(u)\circ \Delta_{>R}\eta)\cdot(\nabla u_2^\sharp + \nabla u_1^\sharp)\;.$$
 Hence, we apply
   Theorem \ref{Thmax} for \eqref{eq:u2sharp} with $b=G(u)\circ \Delta_{>R}\eta$,  $\gamma=\frac{2\alpha+\kappa-1}2$  and $f$ given by $(G(u)\circ \Delta_{>R}\eta)\cdot \nabla u_1^\sharp$, $\cJ$ and other terms on the RHS of \eqref{eq:u2sharp} to obtain
 \begin{equation}\label{Lin:u2sharp}
 	\aligned
 	\|u_2^\sharp\|_{\mL^\infty_T}\lesssim &\sup_{0\leq t\leq T}t^\gamma\|(G(u)\circ \Delta_{>R}\eta)\cdot \nabla u_1^\sharp\|_{L^\infty}+\sup_{0\leq t\leq T}t^\gamma\|\cJ\|_{L^\infty}
 +\|F(u) \Delta_{\leq R}\eta\|_{\mL^\infty_T}
 \\&+\sup_{0\leq t\leq T}t^\gamma\|(F'(u)\prec u)\circ  \Delta_{>R}\eta-F(u)F'(u)\Delta_{>R_1}(\sI(\eta)\circ \eta)\|_{L^\infty}\;.
 \endaligned
 \end{equation}
In the following we consider each term on the RHS of \eqref{Lin:u2sharp}.
Using similar calculation as \eqref{product} we have the following decomposition
\begin{align*}
	&(F'(u)\prec u)\circ  \Delta_{>R}\eta-F(u)F'(u)\Delta_{>R_1}(\sI(\eta)\circ \eta)
	\\&=(F'(u)\prec u^\sharp)\circ  \Delta_{>R}\eta+\text{com}(F'(u),F(u)\prec \sI(\Delta_{>R}\eta),\Delta_{>R}\eta)\\
	&\quad+F'(u) \text{com}(F(u), \sI(\Delta_{>R}\eta), \Delta_{>R}\eta)+F(u)F'(u)(\sI(\Delta_{>R}\eta)\circ \Delta_{>R}\eta)-\Delta_{> R_1}(\sI(\eta)\circ \eta))\;.
\end{align*}
We have the following estimates for each term on the RHS:
\begin{itemize}
	\item Using commutator estimates Lemma \ref{lem:com2} we find
\begin{align*}
&\sup_{0\leq t\leq T}t^\gamma	\|\text{com}(F'(u),F(u)\prec \sI(\Delta_{>R}\eta),\Delta_{>R}\eta)\|_{L^\infty}
\\&+\sup_{0\leq t\leq T}t^\gamma\|F'(u) \text{com}(F(u), \sI(\Delta_{>R}\eta), \Delta_{>R}\eta)\|_{L^\infty}
\lesssim \|u\|_{\mC^{\alpha,\gamma}_T}\|\Delta_{>R}\eta\|_{L^\infty_TC^{\alpha-2}}.
\end{align*}
\item Using paraproduct estimates Lemma \ref{lem:para} we obtain for $\eps>0$
\begin{align*}
		&\sup_{0\leq t\leq T}t^\gamma	\|(F'(u)\prec u^\sharp)\circ  \Delta_{>R}\eta\|_{L^\infty}\lesssim\|\Delta_{>R}\eta\|_{L^\infty_TC^{\alpha-2}}\|u_1^\sharp\|_{\mC_T^{2\alpha,\gamma}}+\|\Delta_{>R}\eta\|_{L^\infty_TC^{-2+\eps}}\|u_2^\sharp\|_{\mC_T^{2-\eps/2,\gamma}}\;.
\end{align*}
\item As $F\in C_b^2$ we have
\begin{align*}
	&\sup_{0\leq t\leq T}t^\gamma\|F(u)F'(u)(\sI(\Delta_{>R}\eta)\circ \Delta_{>R}\eta)-\Delta_{> R_1}(\sI(\eta)\circ \eta)\|_{\mL^\infty_T}
	\\&=\|F(u)F'(u)(\Delta_{\leq R_1}(\sI(\eta)\circ \eta)-\sI(\Delta_{\leq R}\eta)\circ \eta-(\sI(\Delta_{> R}\eta)\circ \Delta_{\leq R}\eta))\|_{\mL^\infty_T}
	\\ &\lesssim 2^{(2\kappa+\eps) R}+2^{(2\kappa+\eps) R_1}\;.
\end{align*}
\end{itemize}
Combining these estimates and using \eqref{bd:u1sharp1}, localizer estimate \eqref{loc}, we arrive at
\begin{align*}
	&\sup_{0\leq t\leq T}t^\gamma\|(F'(u)\prec u)\circ  \Delta_{>R}\eta-F(u)F'(u)\Delta_{>R_1}(\sI(\eta)\circ \eta)\|_{L^\infty}
\\&\lesssim 2^{(2\kappa+\eps) R}+2^{(2\kappa+\eps) R_1}+2^{R(\alpha-1+\kappa)}\|u\|_{\mS_T^{\alpha,\gamma}}+2^{R(-1+\kappa+\eps)}\|u_2^\sharp\|_{\mC_T^{2-\eps/2,\gamma}}+1
\\&\lesssim 2^{(2\kappa+\eps) R}+\|u\|^{\theta_1}_{\mS_T^{\alpha,\gamma\alpha/2+\eps}}+2^{R(\alpha-1+\kappa)}\|u\|_{\mS_T^{\alpha,\gamma}}+2^{R(-1+\kappa+\eps)}\|u_2^\sharp\|_{\mC_T^{2-\eps/2,\gamma}}+1
\\&\lesssim 2^{(2\kappa+\eps) R}+\|u_2^\sharp\|^{\theta_1}_{\mS_T^{\alpha,\gamma\alpha/2+\eps}}+2^{R(\alpha-1+\kappa)}\|u_2^\sharp\|_{\mS_T^{\alpha,\gamma}}+2^{R(-1+\kappa+\eps)}\|u_2^\sharp\|_{\mC_T^{2-\eps/2,\gamma}}+1\,.
\end{align*}
Here $\theta_1=\frac{2\kappa+\eps}{\alpha-2\kappa-\eps}$ and in the second step we used \eqref{c:R1} 
 and in the last step we used \eqref{bd:utime1}  and \eqref{bd:utime}.

 For $	\sup_{0\leq t\leq T}t^\gamma\|\cJ\|_{L^\infty}$ on the RHS of \eqref{Lin:u2sharp} we apply Lemma \ref{Ru} to arrive at
\begin{align*}
	\sup_{0\leq t\leq T}t^\gamma\|\cJ\|_{L^\infty}
&\lesssim \|\Delta_{>R}\eta\|_{L^\infty_TC^{\alpha-2}}(\|u\|_{\mS_T^{\alpha,\gamma}}+\|u_1^\sharp\|_{\mC_T^{2\alpha,\gamma}})+\|\Delta_{>R}\eta\|_{L^\infty_TC^{-2+\eps}}\|u_2^\sharp\|_{\mC_T^{2-\eps/2,\gamma}}
\\&\lesssim 2^{R(\alpha-1-\kappa)}\|u_2^\sharp\|_{\mS_T^{\alpha,\gamma}}+2^{R(-1+\kappa+\eps)}\|u_2^\sharp\|_{\mC_T^{2-\eps/2,\gamma}}+1\,.
\end{align*}
Here in the last step we used \eqref{bd:u1sharp1} and \eqref{bd:utime} and localizer estimate \eqref{loc}.
Moreover, by Lemma \ref{Ru}, \eqref{bdu1} and \eqref{u1} we obtain for $\delta>\kappa$ and close to $\kappa$
\begin{align*}
&\sup_{0\leq t\leq T}t^\gamma\|(G(u)\circ \Delta_{>R}\eta)\cdot \nabla u_1^\sharp\|_{L^\infty}\lesssim \sup_{0\leq t\leq T}t^\gamma\|G(u)\circ \Delta_{>R}\eta\|_{L^\infty}\|u_1^\sharp\|_{\sC^1}
\\&\lesssim (1+ \sup_{t\in[0,T]}(t^\gamma\|u_2^\sharp\|_{{1+}})^{\delta})\|\eta\|_{L^\infty_TC^{-1-\kappa}}\sup_{0\leq t\leq T}t^{\gamma\theta_0}\|u_1^\sharp\|_{C^{1+}}
\\&\lesssim (1+ \sup_{t\in[0,T]}(t^\gamma\|u_2^\sharp\|_{{1+}})^{\delta})\|u_1^\sharp\|^{1-\theta_0}_{C_TC^{1-\kappa}}\sup_{0\leq t\leq T}t^{\gamma\theta_0}\|u_1^\sharp\|^{\theta_0}_{C^{2\alpha}}
\\&\lesssim (1+ \sup_{t\in[0,T]}(t^\gamma\|u_2^\sharp\|_{{1+}})^{\delta})(\|u_2^\sharp\|_{\mS_T^{\alpha,\gamma}}+1)^{\theta_0}2^{R(\alpha-1+\kappa)\theta_0}\;.
\end{align*}
Here $\theta_0=\frac{\kappa+\eps}{2\alpha-1+\kappa}$, $\theta_0+\delta\leq 1$ and we used Lemma \ref{Le32} in the third step and we used \eqref{bdu1} and \eqref{bd:u1sharp1} and \eqref{bd:utime} in the last step.

Substituting the above bounds into the RHS of \eqref{Lin:u2sharp} we use Lemma \ref{Le32} to derive
 \begin{align*}
 \|u_2^\sharp\|_{\mL^\infty_T}
 &\lesssim 1+ \sup_{t\in[0,T]}(t^\gamma\|u_2^\sharp\|_{{1+}})^{\delta/(1-\theta_0)}+\|u_2^\sharp\|_{\mS_T^{2-\eps/2,\gamma}}^{\frac{\alpha\theta_1}{2-\eps/2}}\|u_2^\sharp\|_{\mL^\infty_T}^{\theta_1(1-\frac{\alpha}{2-\eps/2})}
 \\&\quad+2^{R(1+\kappa+\eps)}+2^{R(\alpha-1+\kappa)}\|u^\sharp_2\|_{\mS_T^{2-\eps/2,\gamma}}^{\frac{\alpha}{2-\eps/2}}\|u_2^\sharp\|_{\mL^\infty_T}^{1-\frac{\alpha}{2-\eps/2}}+2^{R(-1+\kappa+\eps)}\|u_2^\sharp\|_{\mS_T^{2-\eps/2,\gamma}}+1\;,
 \end{align*}
 which by Young's inequality implies
 \begin{equation}\label{bd:u2infty}
 	\aligned
 \|u_2^\sharp\|_{\mL^\infty_T}
 & \lesssim 1+ \|u_2^\sharp\|^{\delta/(1-\theta_0)}_{\mS_T^{2-\eps/2,\gamma} }+\|u_2^\sharp\|_{\mS_T^{2-\eps/2,\gamma}}^{\frac{\alpha\theta_1}{(2-\eps/2)(1-\theta_1)+\alpha\theta_1}}
 \\&\quad+2^{R(1+\kappa+\eps)}+2^{(2-\eps/2)R(\alpha-1+\kappa)/\alpha}\|u^\sharp_2\|_{\mS_T^{2-\eps/2,\gamma}}+2^{R(-1+\kappa+\eps)}\|u_2^\sharp\|_{\mS_T^{2-\eps/2,\gamma}}+1\;.
 \endaligned
 \end{equation}

To keep balance we let
$$\|u_2^\sharp\|_{\mS_T^{2-\eps/2,\gamma}}2^{(2-\eps/2)R(\alpha-1+\kappa)/\alpha}\simeq2^{R(1+\kappa+\eps)}$$ which suggest us take
$$2^{R\frac{(2-\alpha+\alpha\kappa-2\kappa+\frac\eps2(3\alpha-1+\kappa))}{\alpha}}\simeq\|u_2^\sharp\|_{\mS_T^{2-\eps/2,\gamma}}+C.$$
Here $C$ is a large constant such that \eqref{bdu} and \eqref{bd:u1sharp2} hold.
Substituting this choice of $R$ into \eqref{bd:u2infty} we obtain
   \begin{align}\label{bd:u2infty1}\|u_2^\sharp\|_{\mL^\infty_T}
\lesssim \|u_2^\sharp\|_{\mS_T^{2-\eps/2,\gamma}}^{\Theta}+1.
 \end{align}
 Here $\theta=\frac{\alpha(1+\kappa+\eps)}{2-\alpha-2\kappa+\alpha\kappa+\frac\eps2(3\alpha-1+\kappa)}$ and $\Theta=\theta\vee \frac{\delta}{1-\theta_0}\vee \frac{\alpha\theta_1}{(2-\eps/2)(1-\theta_1)+\alpha\theta_1}$.

Finally  we use Schauder estimate Lemma \ref{Le11} for $u_2^\sharp$ and  obtain $\|u_2^\sharp\|_{\mS_T^{2-\eps/2,\gamma}}$ could be bounded by the RHS of \eqref{Lin:u2sharp}, which is bounded by \eqref{bd:u2infty1}, and $\sup_{t\in[0,T]}t^\gamma\|G(u)\circ \Delta_{>R}\eta\|_{L^\infty}\|u_2^\sharp\|_{{1+}}$. By Lemma \ref{Ru}, \eqref{bdu1} and paraproduct estimates Lemma \ref{lem:para} we obtain
  \begin{align*}
  	\sup_{t\in[0,T]}t^\gamma\|G(u)\circ \Delta_{>R}\eta\|_{L^\infty}\|u_2^\sharp\|_{{1+}} &\lesssim
  	 \sup_{t\in[0,T]}t^\gamma(1+\|u_2^\sharp\|^{\delta}_{{1+}})\|u_2^\sharp\|_{{1+}}\|\eta\|_{-1-\kappa}
  \\ &\lesssim \|u_2^\sharp\|_{\mS_T^{2-\eps/2,\gamma}}^{(1+\delta)(1+\eps)/2}\|u_2^\sharp\|^{(1+\delta)(1-\eps)/2}_{\mL^\infty_T},
\end{align*}
where in the last step we used Lemma \ref{Le32}. Young's inequality and \eqref{bd:u2infty1} imply
  \begin{align*}
 \|u_2^\sharp\|_{\mS_T^{2-\eps/2,\gamma}}\lesssim \|u_2^\sharp\|_{\mS_T^{2-\eps/2,\gamma}}^{\Theta}+1+\eps\|u_2^\sharp\|_{\mS_T^{2-\eps/2,\gamma}}+\|u_2^\sharp\|_{\mS_T^{2-\eps/2,\gamma}}^{\Theta\frac{(1+\delta)(1+\eps)}{1-\delta+\eps(1+\delta)}}.
 \end{align*}

Take $\eps$ close to $0$, $\alpha$ close to $2/3$ and $\delta$ close to $\kappa$. As we have $-4\kappa^2+7\kappa<1$, we obtain $\Theta\frac{(1+\delta)(1+\eps)}{1-\delta+\eps(1+\delta)}<1$. Thus we obtain
\begin{align*}
 \|u_2^\sharp\|_{\mS_T^{2-\eps/2,\gamma}}\lesssim &1.
 \end{align*}
 By \eqref{bd:utime} and \eqref{bd:u1sharp1} this also implies that
\begin{align*}
 \|u\|_{\mS_T^{\alpha,\gamma}}\lesssim 1,\qquad  \|u^\sharp_1\|_{\mC_T^{2\alpha,\gamma}}\lesssim 1.
 \end{align*}
 Hence the result follows.
\appendix
\renewcommand{\appendixname}{Appendix~\Alph{section}}
\renewcommand{\theequation}{A.\arabic{equation}}
\section{Appendix}

\subsection{Schauder estimates}
Recall that $$\|f\|_{\mS^{\alpha,\beta}_T}:=\sup_{t\in[0,T]}t^\beta\|f(t)\|_{\alpha}+\sup_{0\leq s< t\leq T}s^\beta\frac{\|f(t)-f(s)\|_{L^\infty}}{|t-s|^{\alpha/2}}\,.$$
It is well-known that (see \cite[Lemma 2.10]{MW17})
	\begin{align}\label{Es19}
	\|\Delta_j P_t f\|_{L^\infty}\lesssim_C e^{-2^{2j}t}\|\Delta_jf\|_{L^\infty},\ { j\geq 0}, t\geq 0.
	\end{align}
	
We recall the following result from \cite[Lemma 2.8]{ZZZ20}.
\bl\label{lem:2.8}
Let $T>0$.
\begin{itemize}
\item For any $\theta>0$ and $\alpha\in\mR$, there is a constant $C=C(\alpha,\theta, T)>0$ such that for all $t\in (0,T]$,
\begin{align}\label{E1}
\|P_t f\|_{{\theta+\alpha}}\lesssim_C t^{-\theta/2}\|f\|_{\alpha}.
\end{align}
\item For any $\theta<0$, there is a constant $C=C(\theta, T)>0$ such that for all $t\in (0,T]$,
\begin{align}\label{E4}
\| P_t f\|_{L^\infty}\lesssim_C t^{\theta/2}\|f\|_{\theta}.
\end{align}
\item For any $0<\theta<2$, there  is a constant $C=C(\theta, T)>0$ such that for all $t\in [0,T]$,
\begin{align}\label{E2}
\|P_t f-f\|_{L^\infty}\lesssim_C t^{\theta/2}\|f\|_{{\theta}}.
\end{align}
\end{itemize}
\el

\bl\label{Le11}
For any $\alpha\in(0,2),\gamma\in[0,1)$
\begin{align}\label{EG1}
\|\sI f\|_{\mS^{2-\alpha,\gamma}_T}
\lesssim\|f\|_{\mC^{-\alpha,\gamma}_T}.
\end{align}
\el
\begin{proof}
	 For $t\in(0,T]$,  by \eqref{Es19}, we have for $j\geq0$
	\begin{align*}
	\|\Delta_j\sI f(t)\|_{L^\infty}&\lesssim \int^t_0e^{-(t-s)2^{2j}}s^{-\gamma}\dif s \sup_{t\in[0,T]}t^\gamma\|\Delta_jf\|_{{L^\infty}}\no\\
&\lesssim 2^{j\alpha}\Big( \int^{t/2}_0 2^{-2j}(t-s)^{-1}s^{-\gamma}\dif s +t^{-\gamma}\int_{t/2}^t e^{-(t-s)2^{2j}}\dif s \Big)\|f\|_{\mC^{-\alpha,\gamma}_T}\no
\\&\lesssim 2^{-2j+j\alpha} t^{-\gamma}\|f\|_{\mC^{-\alpha,\gamma}_T}.
	\end{align*}
	When $j=-1$ it is easy to see
\begin{align*}
	\|\Delta_{-1}\sI f(t)\|_{L^\infty}&\lesssim t^{-\gamma}\|f\|_{\mC^{-\alpha,\gamma}_T}.
	\end{align*}
	Thus we obtain
	\begin{align}\label{If}
	\|\sI f(t)\|_{2-\alpha}&\lesssim t^{-\gamma}\|f\|_{\mC^{-\alpha,\gamma}_T}.
	\end{align}
	On the other hand, let $u=\sI f$. For $0\leq t_1< t_2\leq T$, we have
	\begin{align}\label{u2u1}
	u(t_2)-u(t_1)=(P_{t_2-t_1}-I)\sI f(t_1)+\int_{t_1}^{t_2}P_{t_2-s}f(s) \dif s =:I_1+I_2.
	\end{align}
	For $I_1$,  by \eqref{E2} and \eqref{If}, we have
	\begin{align*}
	\|I_1\|_{L^\infty}&\lesssim (t_2-t_1)^{\frac{2-\alpha}{2}}\|\sI f(t_1)\|_{2-\alpha}\\
	&\lesssim (t_2-t_1)^{\frac{2-\alpha}{2}}t_1^{-\gamma}\|f\|_{\mC^{-\alpha,\gamma}_T}.
	\end{align*}
	For $I_2$, by \eqref{E4} we have
	\begin{align*}
	\|I_2\|_{L^\infty}&\lesssim
	\int_{t_1}^{t_2}(t_2-s)^{-\frac{\alpha}{2}}s^{-\gamma}\dif s\|f\|_{\mC^{-\alpha,\gamma}_T}\\
	&\lesssim (t_2-t_1)^{\frac{2-\alpha}{2}}t_1^{-\gamma}\|f\|_{\mC^{-\alpha,\gamma}_T}.
	\end{align*}
These together with \eqref{If} yields \eqref{EG1}.
\end{proof}
We also prove the following result.
\bl\label{Le12}
For any $\alpha\in(0,2),\gamma\in(0,1)$
\begin{align}\label{EG11}
\|\sI f\|_{\mS^{2-\alpha}_T}
\lesssim\|f\|_{\mC^{-\alpha+2\gamma,\gamma}_T}.
\end{align}
\el
\begin{proof}
	For $\gamma=0$ the result follows from \eqref{EG1}. In the following we consider $\gamma\in (0,1)$. For $t\in(0,T]$,  by \eqref{Es19}, we have for $j\geq0$
	\begin{align*}
	\|\Delta_j\sI f(t)\|_{L^\infty}&\lesssim \int^t_0e^{-(t-s)2^{2j}}s^{-\gamma}\dif s \sup_{t\in[0,T]}t^\gamma\|\Delta_jf\|_{{L^\infty}}\no\\
&\lesssim 2^{j(\alpha-2\gamma)} \int^{t}_0 2^{-2j(1-\gamma)}(t-s)^{-(1-\gamma)}s^{-\gamma}\dif s\|f\|_{\mC^{-\alpha+2\gamma,\gamma}_T}\no
\\&\lesssim 2^{-2j+\alpha j} \|f\|_{\mC^{-\alpha+2\gamma,\gamma}_T}.
	\end{align*}
	When $j=-1$ it is easy to see
\begin{align*}
	\|\Delta_{-1}\sI f(t)\|_{L^\infty}&\lesssim \|f\|_{\mC^{-\alpha+2\gamma,\gamma}_T}.
	\end{align*}
	Thus we obtain
	\begin{align}\label{If1}
	\|\sI f(t)\|_{2-\alpha}&\lesssim \|f\|_{\mC^{-\alpha+2\gamma,\gamma}_T}.
	\end{align}
With $u=\sI f$ and writing $u(t_2)-u(t_1)=I_1+I_2$ as in \eqref{u2u1},
 by \eqref{E2}  and \eqref{If1} we have
	\begin{align*}
	\|I_1\|_{L^\infty}&\lesssim (t_2-t_1)^{\frac{2-\alpha}{2}}\|\sI f(t_1)\|_{2-\alpha}\\
	&\lesssim (t_2-t_1)^{\frac{2-\alpha}{2}}\|f\|_{\mC^{-\alpha+2\gamma,\gamma}_T}.
	\end{align*}
For $I_2$, by \eqref{E4}, we have
	\begin{align*}
	\|I_2\|_{L^\infty}&\lesssim
	\int_{t_1}^{t_2}(t_2-s)^{(-\frac{\alpha}{2}+\gamma)\wedge0}s^{-\gamma}\dif s\|f\|_{\mC^{-\alpha+2\gamma,\gamma}_T}\\&\lesssim
	(t_2-t_1)^{\frac{2-\alpha}{2}}\int_{0}^{1}(1-u)^{(-\frac{\alpha}{2}+\gamma)\wedge0}u^{-\gamma}\dif u  \|f\|_{\mC^{-\alpha+2\gamma,\gamma}_T}\\
	&\lesssim (t_2-t_1)^{\frac{2-\alpha}{2}} \|f\|_{\mC^{-\alpha+2\gamma,\gamma}_T}.
	\end{align*}
These together with \eqref{If1} yields \eqref{EG11}.
\end{proof}

The following interpolation inequality will be used frequently, which is an easy consequence of H\"older's inequality and the corresponding definition
 (see e.g. \cite[Lemma A.3]{GH18a} for a discrete version).
\bl\label{Le32}
Let  $\theta\in[0,1]$ and $\alpha,\alpha_1,\alpha_2\in\mR$ satisfying
$ \alpha=\theta \alpha_1+(1-\theta)\alpha_2$.
Then we have
\begin{align}\label{DQ1}
\|f\|_{\alpha}\leq \|f\|_{\alpha_1}^\theta\|f\|_{{\alpha_2}}^{1-\theta}.
\end{align}
Moreover, for any $0<\alpha<\beta\leq2$ with $\theta=\alpha/\beta$ and for any $0\leq\theta\delta_1\leq \delta\leq1$, we also have
\begin{align}\label{AM1}
\|f\|_{\mS^{\alpha,\delta}_T}\lesssim \|f\|_{\mS^{\beta,\delta_1}_T}^{\theta}
\|f\|_{\mL^\infty_T}^{1-\theta}.
\end{align}
\el

\subsection{Paraproduct estimates}

We collect some useful estimates on paraproducts from \cite{GIP15}.

\begin{lemma}\label{lem:para}
For any  $\beta\in\R$,
	\begin{equation}\label{GZ0}
		\|f\prec g\|_{\beta}\lesssim\|f\|_{L^\infty}\|g\|_{{\beta}}\,,
	\end{equation}
	and for any $\alpha<0$ and $\beta\in\mR$,
	\begin{equation}\label{GZ1}
		\|f\prec g\|_{{\alpha+\beta}}\lesssim\|f\|_{\alpha}\|g\|_{\beta}\,.
	\end{equation}
	Moreover, for any $\alpha,\beta\in\mR$ with  $\alpha+\beta>0$,
	\begin{equation}\label{GZ2}
		\|f\circ g\|_{\alpha+\beta}\lesssim\|f\|_{\alpha}\|g\|_{\beta}\,.
	\end{equation}
\end{lemma}

\begin{lemma}\label{lem:com2}
	For any $\alpha\in (0,1)$ and $\beta,\gamma\in \R$ with
	$\alpha+\beta+\gamma>0$ and $\beta+\gamma<0$, there exists a bounded trilinear operator $\mathrm{com}$
	on $C^\alpha\times C^\beta\times C^\gamma$ such that
	\begin{align}\label{FA1}
		\|\mathrm{com}(f,g,h)\|_{{\alpha+\beta+\gamma}}\lesssim
		\|f\|_{\alpha}\|g\|_{\beta}\|h\|_{\gamma}\,,
	\end{align}
	where for smooth functions $f,g,h$,
	$$
	\mathrm{com}(f,g,h):=(f\prec g)\circ h - f(g\circ h)\,.
	$$
\end{lemma}

We now recall the following result from \cite[Lemma A.1]{CC18}.
\begin{lemma}\label{commutator}
	 For any $\alpha\in (0,1)$, $\beta\in \R$, $\delta\geq 0$
	and $T>0$, there is a constant $C=C(\alpha,\beta,\delta,T)>0$
	such that for all
	$f\in C^{\alpha}$, $g\in C^{\beta}$ and $t\in(0,T]$, $j\geq-1$,
	\begin{align*}
	\begin{split}
&	\|\Delta_j P_t(f\prec g)-\Delta_j (f\prec P_tg)\|_{L^\infty}\lesssim_C t^{-\frac\delta2} 2^{-(\alpha+\beta+\delta)j}\|f\|_{\alpha}\|g\|_{\beta}\,.
\end{split}
	\end{align*}
\end{lemma}
Then we could prove the following commutator result for $\sI$.
\begin{lemma}\label{commutator1}
	Let  $\alpha\in (0,1)$, $\beta\in \R,$ $\gamma\in[0,1)$
	There is a constant $C>0$ such that
	\begin{align}
	\|[ \sI, f \prec ]g\|_{\mC_T^{ \alpha+\beta+2,\gamma}}
	&\lesssim_C\|f\|_{\mS_{T}^{\alpha,\gamma}}\|g\|_{L_T^\infty C^\beta}\,.\label{GA44}
	\end{align}
\end{lemma}
\begin{proof}
	We can write
	\begin{align*}
	[\sI, f \prec ]g(t)&=\int^t_0\Big(P_{t-s}(f(s)\prec g(s))-f(t)\prec P_{t-s} g(s)\Big)\dif s
	\\&=\int_0^t\Big(P_{t-s}(f(s)\prec g(s))-f(s)\prec P_{t-s}g(s)\Big) \dif s
	\\&\quad+\int_0^t(f(s)-f(t))\prec P_{t-s} g(s)\dif s
	\\&=:I_1(t)+I_2(t)\,.
	\end{align*}
For $I_1(t)$, by Lemma \ref{commutator}  with $\delta=0,4$ we have
	\begin{align*}
	\|\Delta_jI_1(t)\|_{L^\infty}
	&\lesssim 2^{-(\alpha+\beta)j}\int_0^t\big((4^j(t-s))^{-2}\wedge 1\big) s^{-\gamma}\dif s\|f\|_{\mS_{T}^{\alpha,\gamma}}\|g\|_{L_T^\infty C^\beta}
\\&\lesssim 2^{-(\alpha+\beta+2)j}t^{-\gamma}\|f\|_{\mS_{T}^{\alpha,\gamma}}\|g\|_{L_T^\infty C^\beta}\,.
	\end{align*}
Here for the last integral we used
\begin{align}\label{integral}
&\int_0^t\big((4^j(t-s))^{-2}\wedge 1\big) s^{-\gamma}\dif s\no
\\&\lesssim 2^{-2j}\int_0^{t/2}(t-s)^{-1} s^{-\gamma}\dif s+ t^{-\gamma}\int_{t/2}^t \big((4^j(t-s))^{-2}\wedge 1\big)\dif s\no
\\&\lesssim 2^{-2j} t^{-1} \int_0^{t/2} s^{-\gamma}\dif s+t^{-\gamma}2^{-2j}\int^{2^{2j}t/2}_0 \big(s^{-2}\wedge 1\big)\dif s
\\&\lesssim 2^{-2j} t^{-\gamma}\,.\no
\end{align}

	For $I_2(t)$, for any $\delta>0$, note that by \eqref{GZ0} and \eqref{E1}
	\begin{align*}
	&\|\Delta_j((f(s)-f(t))\prec P_{t-s} g(s))\|_{L^\infty}\\
	&\quad\lesssim 2^{-(\delta+\beta)j}\|f(s)-f(t)\|_{L^\infty}\|P_{t-s} g(s)\|_{\delta+\beta}\\
	&\quad\lesssim 2^{-(\delta+\beta)j}(t-s)^{\frac{\alpha-\delta}2}s^{-\gamma}\|f\|_{\mS_{T}^{\alpha,\gamma}}\|g(s)\|_{{\beta}}\,,
	\end{align*}
which implies that (choosing $\delta=\alpha, \alpha+4$)
	\begin{align*}
       \|\Delta_jI_2(t)\|_{L^\infty}
&       \lesssim 2^{-(\alpha+\beta)j}\int_0^t(1\wedge ((t-s)4^j)^{-2})s^{-\gamma}\dif s
\|f\|_{\mS_{T}^{\alpha,\gamma}}\|g\|_{L_T^\infty C^\beta}\\
&\lesssim2^{-(\alpha+\beta+2)j}\|f\|_{\mS_{T}^{\alpha,\gamma}}\|g\|_{L_T^\infty C^\beta}\,.
 	\end{align*}
Here for the last integral we used \eqref{integral}.
Thus we complete the proof.
\end{proof}

\end{document}